\documentclass[10pt,a4paper]{amsart}  % 26.02.2026
\usepackage{amssymb}
\usepackage{amsmath,amsfonts,amsthm}
\usepackage{amstext,latexsym,bm}
\usepackage{graphicx}
\usepackage{rotating} % for rotation

\usepackage{xcolor}
 
\usepackage{amscd}
 
\usepackage{amssymb}
\usepackage{amstext}
\usepackage{mathtools}
\usepackage{bbm}

\newtheorem{theorem}{Theorem}

\newtheorem{proposition}[theorem]{Proposition}
\theoremstyle{definition}

\newtheorem{remark}[theorem]{Remark}
\newtheorem{definition}[theorem]{Definition}

\numberwithin{equation}{section}
\numberwithin{theorem}{section}

\begin{document}
\author{Marc Taberner-Ortiz and Manfred Denker }
\title{Remarks on stationary GARCH processes under heavy tail distributions} 

\maketitle

\begin{abstract}  Let $(X_n)_{n\in \mathbb Z}$ be a GARCH process with $E(X_0^4)<\infty$, and let $\mu_n$ denote the distribution of $\frac 1{{\sqrt n}}\sum_{i=1}^n [X_i^2-\mathbb E(X_0^2)]$.
We derive a numerical approximation of $\mu_n$ when $x_1,...,x_n$ are observed. This yields the derivation of confidence intervals for $\mu= E(X_0^2)$ and we investigate 
the accuracy of these confidence intervals in comparison with standard ones based on normal approximation. Moreover, when the innovation process has heavy tail distribution, we improve the method using a new resampling method.
\end{abstract}

\section{Introduction}\label{sec:1}
GARCH processes describe
an important class of stationary time series, in particular, for applications in econometrics and forecasting (see \cite{TB} and \cite{RFE} for specific references). In order to set up the notation recall that a stationary real valued process $(X_n)_{n\in \mathbb Z}$ is called a GARCH($p,q$) ($p,q\in\mathbb N$) process if
there exists an i.i.d. process $(\xi_n)_{n\in \mathbb Z}$, a positive real valued process $(\sigma_n)_{n\in \mathbb Z}$ and non-negative  real numbers $a_0,a_1,...,a_p$, $b_1,...,b_q$ such that
\begin{eqnarray}
&& E\xi_0=0\qquad E\xi_0^2=1\label{eq:1.1}\\
&&  X_n= \sigma_n\xi_n\label{eq:1.2}\\
&& \sigma_n^2= a_0+\sum_{i=1}^p a_iX_{n-i}^2 + \sum_{j=1}^q b_j \sigma_{n-j}^2\label{eq:1.3}
\end{eqnarray}
Conditions for existence of such stationary processes can be found in \cite{L}.
In case $p=q=1$ the condition for stationarity are described in \cite{N}: The condition 
\begin{equation}\label{eq:1.4}
-\infty < \mathbb E(\log(a_1\xi_0+b_1)]<0
\end{equation}
is necessary and sufficient for the existence of stationary GARCH($1,1$) processes. In this case we have
\begin{equation}\label{eq:1.5a}
 \sigma_n^2= a_0\left(1+ \sum_{j=1}^\infty \prod_{i=1}^j (a_1\xi_{n-i}^2+b_1)\right).
 \end{equation}

In this note we restrict to the case $p=q=1$.  Presumably, the general case can be handled using a similar approach, the mathematical part seems to be  more complicated.  Instead, we extend the class of processes to augmented GARCH($1,1$) processes $(X_n)_{n\ge 1}$ where equation (\ref{eq:1.3}) is replaced by:
\newline
There exist a coordinatewise increasing, measurable function $\varphi:\mathbb R^2\to \mathbb R$  and a measurable strictly increasing function $\Lambda:\mathbb R_+\to\mathbb R$  such that
\begin{equation}\label{eq:1.5}
\Lambda(\sigma_n^2)= \varphi(\Lambda(\sigma_{n-1}^2), \xi_{n-1}^2).
\end{equation}
Note that the  recursion formula for $\sigma_n^2$ only involves $\xi_{n-1}$ and not $X_{n-1}$.

Returning to the case of a GARCH($1,1$) process it has been shown by Berkes et al. \cite{BHH} that the statistics
$$ T_n=\frac 1{\sqrt{n}} \sum_{k=1}^n[X_k^2-E(X_0^2)]$$
converges weakly to some centered normal distribution. In this note we give a direct simple proof. Our approach as well permits to show that 
\begin{equation}\label{eq:1.6}
\lim_{n\to\infty}  \frac 1{\log n}\sum_{k=1}^n \frac 1k \mathbbm 1_{\{T_k\le t\}}-\mathbb P(T_n\le t) =0.
\end{equation}
This result permits to construct new confidence intervals which will be numerically compared to those constructed from the Berkes et al. result. In order to obtain some information on the speed of convergence in (\ref{eq:1.6}) we suggest to show the limit approaching 0 when the left hand side is divided by $\mathbb P(T_n\le t)$.

 We derive equation (\ref{eq:1.6}) in Section \ref{sec:3}, the case of an augmented GARCH($1,1$) process is proved similar. Numerical results are deferred to Section \ref{sec:5}. As is shown there, the confidence levels obtained from (\ref{eq:1.6}) are hardly met when the innovation process has heavy tails. Therefore we use a new resampling method with stable distributions to improve the cover probability, its theoretical background is  contained in Section \ref{sec:4}.

The motivation for the present study may be seen in the imprecise  confidence intervals obtained by normal approximation  using the Berkes et al. result \cite{BHH}. This has been observed in \cite{MTO} and we present similar  simulations for a variety of innovation's process distributions here. The proposed distributions are: the standard normal distribution $N$, the $t$-distribution with distinct degrees of freedom and Pareto distributions, the representation $P(\alpha,x_m)$ is chosen to represent the Pareto distribution with scale parameter $x_m$ and shape parameter $\alpha$.\\

\begin{table}[h]
    \centering
    \begin{tabular}{ccccccc}
        &Sample size & 100 & 300 & 500 &  1000 & 1500 \\ 
        \hline
         &$N$  & 0.882& 0.922 & 0.928 & 0.952 & 0.947  \\
          &$t_8$& 0.858 &0.900& 0.922 & 0.927 & 0.928 \\
          &$t_6$ & 0.857 & 0.891 & 0.905& 0.907 & 0.925\\
          &$P(8,1)$ & 0.726 & 0.822 & 0.846 & 0.867 & 0.883 \\ 
          & $P(6,1)$& 0.666  & 0.747 & 0.791 & 0.808 & 0.844 \\
          \hline
              \end{tabular}
    \caption{Coverage rate for distinct sample sizes obtained from the normal approximation method.  The simulations represent   a 95\%
    coverage rate where 1000 simulations were performed of a GARCH(1,1) process with parameters $a_0=a_1=b_1=0.1 $.}%\label{tab:placeholder0}
\end{table}

\section{Local and global approximations}\label{sec:2}
The results for the stationary GARCH($1,1$) processes and their augmented modifications are stated in this section. Their proofs are to be found in the next section.
\begin{theorem}\label{theo:2.1}\cite{BHH} Let $(X_n)_{n\in \mathbb Z}$ be a stationary GARCH($1,1$) process, $X_0\in L_2$ and $\xi_0\in L_4$. Then there exists $\tau^2\ge 0$ such that the statistics
\begin{equation}\label{eq:2.1} T_n=\frac 1{\sqrt{n}} \sum_{k=1}^n [X_k^2- \mathbb E(X_0^2)]
\end{equation}
 converge weakly to a normal distribution with expectation 0 and variance $\tau^2$.
\end{theorem} 
\begin{remark} It is well known that
$$\tau^2= a_0^2\frac{1+a_1+b_1}{(1-a_1-b_1)^2} \left(\frac{\mathbb E(\xi_0^4)(1+a_1-b_1) + 2b_1}{1-\rho^2}- \frac 1{1-a_1-b_1}\right)$$
where $\rho^2=\mathbb E(a_1\xi_0^2+b_1^2)<1$.
\end{remark}
\begin{theorem}\label{theo:2.2}
Let $(X_n)_{n\in \mathbb Z}$ be a stationary GARCH($1,1$) process, $X_0\in L_2$ and $\xi_0\in L_4$. Then the distributions of the statistics
$$ T_n=\frac 1{\sqrt{n}} \sum_{k=1}^n [X_k^2- \mathbb E(X_0^2)]\qquad (n\ge 1)$$
satisfy the following almost sure limit  for any $t\in \mathbb R$
$$ \lim_{n\to\infty} \left|\mathbb P(T_n\le t) - \frac 1{\log n}\sum_{k=1}^n \frac 1k \mathbbm 1_{\{T_k\le t\}}\right | =0.$$
\end{theorem}
\begin{theorem}\label{theo:2.3} Let $(X_n)_{n\in \mathbb Z}$ be an augmented GARCH($1,1$) process which satisfies (\ref{eq:1.5}) and 
$$ \tau^2=\sum_{k=1}^\infty \mbox{\rm Cov}(X_1^2, X_k^2)<\infty.$$  Then the distributions of $T_n$ defined in (\ref{eq:2.1}) converge weakly to  the normal distribution with zero expectation and variance $\tau^2$.
\end{theorem}
\begin{theorem}\label{theo:2.4} Let $(X_n)_{n\in \mathbb Z}$ be an augmented GARCH($1,1$) process as in Theorem \ref{theo:2.3}. Then for any $t\in \mathbb R$
$$ \lim_{n\to\infty} \left|\mathbb P(T_n\le t) - \frac 1{\log n}\sum_{k=1}^n \frac 1k \mathbbm 1_{\{T_k\le t\}}\right | =0\qquad {\rm a.s.}.$$
\end{theorem}

\section{Association}\label{sec:3}

The proofs of the theorems in Section \ref{sec:2} rely on association. For completeness we recall basic properties of associated processes. As a general reference we refer to \cite{PR}. 
\begin{definition}\label{def:3.1} 
\begin{enumerate}
\item A finite collection of random variables $\{X_1, X_2,..., X_n\}$ is called associated if 
\begin{equation}\label{eq:3.1}
\mbox{\rm Cov}(f(X_1,X_2,...,X_n),g(X_1,X_2,...,X_n)) \ge 0.
\end{equation}
 for any pair of measurable functions $f,g:\mathbb R^n\to\mathbb R$ which are coordinatewise nondecreasing, provided the covariances are well defined.
 \item A process is called associated if any finite subcollection of random variables is associated.
 \end{enumerate}
\end{definition}
\begin{remark}\label{rem:3.1}
We need the following assertions on associated random variables.
\begin{enumerate}
\item A sequence of independent random variables is associated
(see \cite{PR}, Section 1.1, page 3).
\item Non-decreasing functions of associated random variables are associated (\cite{PR}, Section 1.2, page 4).
\item A countable family of finite  sums of associated random variables is associated. (see \cite{PR}, Section 1.1). Then, by (2), the same assertion holds for finite products provided the associated sequence is strictly positive.
\item Let $\{X_1, X_2,..., X_n\}$ be an associated collection of random variables, and let $f_i(x_1,x_2,...,x_n)$, $ i \in \{1,2,...,m\}$, be coordinate-wise increasing measurable functions, then the collection $\{f_i(X_1,X_2,...,X_n):\ 1\le i\le m\}$ is a collection of associated random variables. This is shown by direct computation.
\end{enumerate}
\end{remark}

\begin{proof}[ Proof of Theorems \ref{theo:2.1}--\ref{theo:2.4}]

First observe that for $\mu=EX_0^2$
$$\sum_{k\ge 1} E(X_1^2-\mu)(X_k^2-\mu) <\infty.$$
This has been shown in \cite{BHH}.
It follows that Theorem \ref{theo:2.1} holds by Newmans central limit theorem for associated random sequences (see \cite{NW}) provided it can be shown that the sequence is associated. 

Under the same condition Theorem \ref{theo:2.2}  follows from Theorem 2 in \cite{PS}.

In order to complete the proof of the first two theorems it is left to show that the sequence $X_n^2$, $n\in \mathbb Z$, is associated (so $X_n^2-\mu$ is associated by Remark \ref{rem:3.1}(2)).

\begin{proposition}\label{prop:3.1} 
A stationary GARCH($1,1$) process as in Theorems \ref{theo:2.1} or \ref{theo:2.2} is associated.
\end{proposition}
\begin{proof}
We first show that $\sigma_n^2$ form an associated sequence.

Notice that the square of the innovation processes $(\xi_n^2)_{n\in \mathbb Z}$ are independent, hence by Remark \ref{rem:3.1}, also associated. Since $a_1>0$, the function $f(x)=a_1x+b_1$ is increasing, thus the collection $\{a_1\xi_n^2 + b_1: n \in \mathbb Z\}$ is associated . The partial products of an associated collection of positive random variables are also associated by Remark \ref{rem:3.1}(2), meaning that $\{ \prod_{i=1}^j(a_1\xi_{n-i}^2 + b_1) : n\in \mathbb Z, j \in \mathbb N\}$ is an associated collection of random variables. Since $\mathbb E[\sum_{j=1}^\infty\prod_{i=1}^j(a_1\xi_{n-i}^2 + b_1)] < \infty$, the collection  $\{\sum_{j=1}^\infty\prod_{i=1}^j(a_1\xi_{n-i}^2 + b_1): n\in \mathbb Z\}$ is associated by Remark \ref{rem:3.1}(3). As $a_0> 0$, it follows easily that  $\{\sigma_n^2: n\in \mathbb Z\}$ is associated.

We next show by induction that the family $\{(\sigma_n,X_n^2): n\in \mathbb N\}$ is associated.

For $n=1$ we have that $\sigma_1^2$ is measurable  with respect to the $\sigma$-field generated by all $\xi_k$ with $k\le 0$ (cf. (\ref{eq:1.5a}). Therefore  $\sigma_1^2$ and $\xi_1^{2}$ are independent, hence associated. The functions
$$f(x,y)= xy\mathbbm 1_{\{x,y\ge 0\}}\quad g(x,y)=x \quad h(x,y)=y$$ are coordinatewise nondecreasing on $\mathbb R^2$, so by Remark \ref{rem:3.1}(4) 
$$(\xi_1^{2},\sigma_1^2,X_1^2)=(g(\xi_1^2,\sigma_1^2),h(\xi_1^2,\sigma_1^2),f(\xi_1^2,\sigma_1^2))$$ 
is associated.

Assume now that the family $\{\xi_k^2,\sigma_k^2, X_k^2: 1\le k\le n-1\}$ is associated.

We shall show that this holds for the the family $\{\xi_k^2,\sigma_k^2, X_k^2: 1\le k\le n\}$. We proceed similar to the case $n=1$. The function 
$$f(x_1^2,...,x_{n-1}^2,y_1^2,...,y_{n-1}^2,z_1^2,...,z_{n-1}^2)=
a_0+a_1y_{n-1}^2+bz_{n-1}^2$$
is nondecreasing. Therefore 
$$\{\xi_k^2,\sigma_k^2, X_k^2: 1\le k\le n-1\}\cup 
\{ f(\xi_1^2,...,\xi_{n-1}^2,\sigma_1^2,...,\sigma_{n-1}^2, X_1^2,...,X_{n-1}^2)=\sigma_n^2\}$$
is associated. Moreover, since $\xi_n^2$ is independent of this family we can add it without changing associativity. Finally, we are able to add $\sigma_n^2\xi_n^2$ to conclude the induction step.
\end{proof}

The proof of Theorem \ref{theo:2.3} and \ref{theo:2.4} are similar and left to the reader.
\end{proof}

\section{Resampling}\label{sec:4}

This section extends and adds to the previous results in as much as it provides a continuum of estimation procedures, essentially parametrized by a parameter $1<\mathfrak{p}<2$. They have different properties as  the size and the cover probabilities of the confidence intervals  decrease as $\mathfrak{p}$ increases. The method is based on resampling with a standard stable distribution of order $\mathfrak{p}$.

\subsection{Stable integrals}
We begin recalling the construction in \cite{RW} to obtain {\it stable integrals}
$$ \int_0^t \mathbb X(\omega,s)dM(\omega,s).$$

Let $(\Omega,\mathcal F,\mathbb P)$ be a probability space and $\mathcal F_t\subset \mathcal F$ ($t>0$) be a right continuous increasing family of $\mathbb P$-complete sub-$\sigma$-fields of $\mathcal F$. For $0<{\mathfrak p}<2$ an $(\mathcal F_t)_{t>0}$-adapted process $(M(\cdot,t))_{t\ge 0}$ is called a ${\mathfrak p}$-stable motion if
\newline(1) For all $0\le s<t$ and $y\in \mathbb R$
$$ \mathbb E\{\exp(iy(M(t)-M(s)))|\mathcal F_s\}= \exp(-(t-s)|y|^{\mathfrak p}).$$
(2) For all $\omega\in \Omega$ the path 
$$ \mathbb R_+\ni t\mapsto M(\omega,t)$$
belongs to $D([0,\infty))$.

The space $L^{\mathfrak p}(L^{\mathfrak p})$ is defined as the space of all real valued measurable processes $\mathbb X=(X(\cdot,t))_{t\ge 0}$ on $\Omega\times [0,\infty)$ which are adapted to $(\mathcal F_t)_{t\ge 0}$ and satisfy for each $T>0$
$$ \|\mathbb X\|_{{\mathfrak p},T}= \left( \mathbb E\int_0^T|\mathbb X(\cdot,s)|^{\mathfrak p} ds\right)^{1/{\mathfrak p}} <\infty.$$
A {\it simple process} has the form
$$ \mathbb X(\omega,t)= \phi_0(\omega)\mathbbm 1_{\{0\}}(t) +\sum_{k=0}^\infty\phi_k(\omega) \mathbbm 1_{(t_k,t_{k+1}]}(t),$$
where $0=t_0<t_1<...<t_k<...\to\infty$ and $\phi_k$ is $\mathcal F_{t_k}$-measurable for each $k\ge 0$.  Its integral with respect to $M(t)$ is the defined by
$$\int_0^t \mathbb X(\omega,s) dM(\omega,s)= \sum_{k=0}^{n-1}\phi_k(\omega)(M(\omega,t_{k+1})-M(\omega,t_k))+ \phi_n(\omega)(M(\omega,t)-M(\omega,t_n))$$
where $t_n\le t<t_{n+1}$ determines $n$. This integral is clearly a process with sample paths in $D([0,\infty))$.

Next let $\Lambda^{\mathfrak p}(L^\infty)$ denote the space of all $(\mathcal F_t)$-adapted processes $\mathbb Y(\cdot,t)$, $t\ge 0$, with almost all sample paths in $D([0,\infty))$  such that for every $T>0$ the weak $L^{\mathfrak p}$-norm
$$ \left(\sup_{\lambda>0}\lambda^{\mathfrak p} \mathbb P\left\{\sup_{t\le T}|Y(\cdot,t)|>\lambda\right\}\right)^{1/{\mathfrak p}}<\infty.$$

It has been shown in \cite{RW}, Theorem 2.1, that 
$$ \mathbb X\mapsto \int \mathbb X dM$$
defined for simple processes $\mathbb X$ extends to an isomorphic embedding of $L^{\mathfrak p}(L^{\mathfrak p})$ into $\Lambda^p(L^\infty)$. Thus for $\mathbb X\in L^{\mathfrak p}(L^{\mathfrak p})$  the integral $\int \mathbb X dM\in \Lambda^{\mathfrak p}(L^\infty)$ is well defined.
In particular, we shall write $I(h):=\int h(\mathbb X) dM$ to denote the integral when $h$ is a non-random function $h:[0,1]\to\mathbb R$ in $L^{\mathfrak p}([0,1])$.

\subsection{Resampling of GARCH processes}

Let $(X_n)_{n\in \mathbb Z}$ be a GARCH process as before with innovation process $(\xi_n)_{n\in \mathbb Z}$ and constants $a_0$, $a_1,...,a_p$ and $b_1,..., b_q$.
Let $F_X$ denote the continuous distribution function of the variables $X_n$, which is independent of $n$  since the process is stationary and ergodic. Then $F_X^{-1}(U)$ has the same distribution as $X_n$ for every $n$ if $U$ has uniform distribution. Hence $U_n= F_X(X_n)$ is a stationary process with uniform one dimensional marginals. 

Denote by $(Y_n)_{n\ge 1}$ a centered stable process of order $\mathfrak{p}\in (1,2)$ with independent $Y_i$, which in addition is also independent of $(X_n)_{n\in \mathbb Z}$.
The resampled  empirical process is defined by
$$ F_n(x):=\frac 1{n^{1/\mathfrak{p}}}\sum_{i=1}^n \mathbbm 1_{\{F_X(X_i)\le x\}} Y_i
\qquad x\in [0,1].$$
Define $\mathcal F_t$ to be the $\sigma$-field generated   by the random variables $F_n(s)$ for $n\ge 1$ and $s\le t$.
\begin{theorem}\label{theo:4.1} There exist a stable motion $M(\cdot,t)$ ($0\le t\le 1$) such that $F_n$ converges weakly to $M$ in the Skorohod topology on $D([0,1])$.
\end{theorem}
\begin{proof} The proof follows in the same way a the corresponding one in \cite{DDW} just using the ergodic theorem for the $X_n$ instead of the law of large numbers for i.i.d. sequences. In particular this applies to page 7  in \cite{DDW} to show the convergence of finite dimensional distributions and to Lemma 3.2 for the tightness argument.
\end{proof}

Its discrete analog to $I(h)$ with respect to the resampled empirical process $F_n$ is defined as
$$ I_n(h):=\int h(F_X^{-1}(u)) dF_n(u)= \frac 1{n^{1/{\mathfrak p}}}\sum_{k=1}^n h(X_i)Y_i.$$ 

In what follows we take $h\circ F_X^{-1}-EX_1^2\in L^r([0,1])$ for some $r>\mathfrak p$ where $h(x)=x^2$.

\begin{theorem}\label{theo:4.2} Let $(Y_n)_{n\ge 1}$ have a stable distribution with Fourier transform  $\exp\{|t|^\mathfrak{p}$, $t\in \mathbb R$. For the processes $(X_n)_{n\in \mathbb Z}$ as above, it holds:
$$ \lim_{n\to\infty} d(\mathcal L(I_n(h)),\mathcal L(I(h\circ F_X^{-1}))) =0$$
where $d$ stands for a metric of weak convergence, and 
$$ \lim_{N\to\infty} \frac 1{\log N}\sum_{n=1}^\infty \frac 1n \mathbbm1_{\{ I_n(h)\le x\}} = \mathbb P(I(h\circ F_X^{-1})\le x)\qquad \mbox{\rm a.s.}.$$
\end{theorem}
\begin{proof} For $h_0\circ F_X^{-1}= \mathbbm 1_{[\alpha,\beta]}$, $0\le \alpha <\beta\le 1$, we have
$$I_n(h_0)=\frac 1{n^{1/{\mathfrak p}}}\sum_{k=1}^n \mathbbm 1_{ [\alpha,\beta]}(F(X_k)) Y_k\to M(\beta)-M(\alpha)= \int \mathbbm 1_{[\alpha,\beta]}(t) dM(t).$$
For simple functions $h_1(x)=\sum_{l=1}^L c_l \mathbbm 1_{[\alpha_l,\alpha_{l+1})}(x)$ the same argument together with the independence of increments of  $M$ shows that  $I_n(h_1)\to I(h_1)$.
In order to show the first part note that Theorem 2.1 in \cite{DDW} holds as well in the present context so that there is a constant $C$ so that for $n\ge 1$
$$ \|I_n(h_1)\|_{L_p(\mathbb P)}\le C \|h_1\|_{L_r([0,1])}.$$

The second part uses Theorem 1 in \cite{BD} for independent processes which says that it is sufficient to 
show that $\sup_{n\ge 1} E|I_n(h_2)| <\infty$.  Since the process $(X_n)_{n\in \mathbb Z}$ is only stationary it cannot be applied directly. However, for a function of the form
$$ h_2(x)=\sum_{k\in \mathbb Z} (c_k^2-EX_1^2) \mathbbm 1_{[\alpha_k,\alpha_{k+1})}(x)$$
where $c_k\in \mathbb R$ and $\alpha_k<\alpha_{k+1}$ defines a partition of $\mathbb R_+$ allows to write
$$ I_n(h_2)=\frac 1{n^{1/\mathfrak{p}}}\sum_{k\in \mathbb Z} (c_k^2 -EX_1^2)\sum_{i=1}^n \mathbbm 1_{[\alpha_k,\alpha_{k+1})}(X_i) Y_i.$$
The right hand side can be written as a weighted sum of i.i.d. random variables. Let $J_k=\{i: X_i\in [\alpha_k,\alpha_{k+1})\}$, then
$$ I_n(h_2) = \frac 1{n^{1/\mathfrak{p}}}\sum_{k\in \mathbb Z} (c_k^2-EX_1^2) |J_k|^{1/\mathfrak{p}} \frac 1{|J_k|^{1/\mathfrak{p}}}\sum_{i\in J_k} Y_i = \sum_{k\in \mathbb Z} (c_k^2-EX_1^2) |J_k|^{1/\mathfrak{p}} Z_k$$
where $Z_k$ has the same distribution as $Y_1$ by assumption. The Fourier transform  becomes $\exp (|t|^\mathfrak{p} \sum_k (|c_k^2-EX_1^2|^{\mathfrak{p}} \left| J_k\right|)n^{-1})$, hence the variable has the same distribution as $\left(\sum_k|c_k^2-EX_1^2|^{\mathfrak{p}} \frac{|J_k|}{n}\right)^{1/\mathfrak{p}} Z_1$. Thus the second claim holds for all such $h_2$.

Let $\epsilon>0$. Choose a function of the form of $h_2$ with $X_i-EX_1^2\le c_k^2$ for $i\in J_k$ and $\sup_t P( I(h_2)\le t)-P(I(h)\le t) <\epsilon$. Then by monotonicity
$$\frac 1{\log N}\sum_{n=1}^N \frac 1 n \mathbbm 1_{\{I_n(h)\le t\}}\le 
\frac 1{\log N}\sum_{n=1}^N \frac 1n \mathbbm 1_{\{I_n(h_2)\le t\}},$$
whence 
$$\lim_{N\to\infty} \frac 1{\log N}\sum_{n=1}^N \frac 1n \mathbbm 1_{\{I_n(h)\le t\}}
\le P(I(h_2)\le t) \le P(I(h)\le t)+\epsilon \qquad a.s.$$
This together with the lower bound (obtained similarly) shows the second claim.
\end{proof}

\section{Inference on the variance of GARCH(1,1) processes}\label{sec:5}
\subsection{Estimation using Theorem \ref{theo:2.4}}\label{sec:5.1}
The confidence interval for $\mu=EX_1^2$ is provided using the standard method: If the sample has size $n\in \mathbb N$ then
$$ L_\alpha \le \frac{\sum_{i=1}^n X_i^2- n \mu}{\sqrt{n}} \le U_\beta$$
defines the confidence interval for $\mu$ with error probability $\alpha+\beta$. Contrary to the standard procedure to determine the bounds $L_\alpha$ and $U_\beta$ by the asymptotic distribution, use 
Theorem \ref{theo:2.4} instead. This needs to replace $\mu$ by a suitable value, for example by $\frac 1n \sum_{i=1}^n X_i^2$. In case of a symmetric distribution the median also is a suitable choice. 

However, as the simulations in Section \ref{sec:6} show this may lead to unreliable confidence intervals. This is due to the fact that the empirical mean is further away from the true mean than the length of the estimated confidence interval.  This leads to the problem of finding methods providing larger intervals, and hence Theorem \ref{theo:4.2} is a natural  choice to accomplish this.

The estimation of the distribution by logarithmic averages is certainly  slow. However, one may show
$$ \lim_{N\to\infty} \frac 1{\log N} \sum_{n=1}^N \frac 1{n P(S_n\le t)} \mathbbm 1_{\{S_n\le t\}} =1,$$
which implies that the speed of convergence in the tails is much faster. (Such a result appears in \cite{JW} for negatively associated sequences.)
In order to improve the estimation of the unknown distribution note that replacing $\log N$ by $\sum_{k=1}^N\frac 1k$ makes the approximation function a distribution function. Moreover the a.s. result holds as well when neglecting the first terms in the sum $\sum_{n=1}^N$ and it also holds when shifting the sample which a.s. is again a sample approximating the unknown distribution, because the distribution is invariant under the natural action of $\mathbb Z$.

\subsection{Estimation using Theorem \ref{theo:4.2}}\label{sec:5.2}

Take $(Y_i)_{i\ge 1}$ to be an i.i.d. sequence of stable random variables with Fourier transform $\exp(it-|t|^\mathfrak{p})$. Then $Y_i-1$ satisfies the hypothesis needed to apply Theorem \ref{theo:4.2}. Therefore
$$ I_n(h)= \frac 1{n^{1/\mathfrak{p}}} \sum_{i=1}^n (X_i^2-EX_1^2)(Y_i-EY_1)$$
converges weakly to $I(h)$ and the almost sure version holds as well.
Since by definition the process $\{X_n^2-EX_1^2\}_{n\in \mathbb Z}$ is orthogonal the variance of $\sum_{i=1}^n X_i^2$ grows linearly in $n$ and hence by standard arguments in probability
$$ \frac 1{n^{1/\mathfrak{p}}}\sum_{i=1}^n (X_i^2-EX_1^2) EY_1 \to 0.$$
Therefore
\begin{eqnarray*}
&& \frac 1{n^{1/\mathfrak{p}}}\sum_{i=1}^n [X_i^2-EX_1^2][Y_i-EY_1] \\
&&\qquad = 
\frac 1{n^{1/\mathfrak{p}}}\sum_{i=1}^n [X_i^2-EX_1^2]Y_i + \frac 1{n^{1/\mathfrak{p}}}\sum_{i=1}^n [X_i^2-EX_1^2]EY_1 \\
&&\qquad = \frac 1{n^{1/\mathfrak{p}}}\sum_{i=1}^n [X_i^2-EX_1^2]Y_i + o(1) \quad{\rm a.s.}
\end{eqnarray*}
It follows that the sequence $\sum_{i=1}^n [X_i^2-EX_1^2]Y_i$ satisfies the weak convergence and the almost sure weak convergence results. Thus we can get the confidence interval for $\mu=EX_1^2$ by choosing empirical quantiles $z_\alpha$ and $z_\beta$ from the almost sure approximation (with an a priori estimate for $EX_1^2$):
$$ \frac{\frac 1n \sum_{i=1}^n X_i^2 - z_\alpha n^{-(\mathfrak{p}-1)/\mathfrak{p}}}{\frac 1n \sum_{i=1}^n Y_i}\le \mu\le 
\frac{\frac 1n \sum_{i=1}^n X_i^2 - z_\beta n^{-(\mathfrak{p}-1)/\mathfrak{p}}}{\frac 1n \sum_{i=1}^n Y_i}.$$

In order to apply this procedure the same advices as in Section \ref{sec:5.1} apply. In addition, the sequence $Y_i$ is repeatedly calculated until its mean differs from $1$ by some prescribed bound. Moreover, the order of the stable distribution  can be chosen arbitrarily, the larger $\mathfrak p$ the shorter the confidence interval but the smaller the coverage rate (see the simulations in the next section).
%\end{proof}

\section{Numerical simulations}\label{sec:6}
In this section, one can find numerical simulations regarding the results shown in Sections \ref{sec:2}, \ref{sec:4} and \ref{sec:5}. Different distributions will be considered as innovations for GARCH processes; these selected distributions will present varied tail decay behaviours. It will be observed that, via decreasing the parameter $\mathfrak{p}$, the confidence intervals adjust better distributions with fatter tails.\\

First, we will see the behaviour of the known asymptotic limit  (using Theorem \ref{theo:2.2}); this can be interpreted as the limiting case when $\mathfrak{p} \to 2$.  Then, the same distributions will be tested using the resampling method with values $\mathfrak{p} \in \{ 1.8,1.65,1.5,1.35\}$. 
\\

The simulations consist of GARCH($1,1$) processes of length 600 with parameters $a_0 =a_1 = b_1 = 0.1$. When applying the limits, the sum will start in the fifth value, the first 4 will be neglected, moreover, 5 shifts of length  100 will be considered for all the simulations. For the used stable distributions $Y_i$, the distribution will be deemed adequate whenever that the empirical mean does not differ from 1 too much, that is, whenever, $|1 - \sum_{i} Y_{i} / n| < 0.2 $. \\

The following table contain the chosen distributions for the innovations processes, the confidence interval length and coverage rate for the distinct proposed methods. The proposed distributions are: the standard normal distribution $N(0,1)$, the $t$ distribution with distinct degrees of freedom and Pareto distributions, the representation $P(\alpha, x_{m})$ is chosen to represent the Pareto distribution with  scale parameter $x_{m}$ and the shape parameter $\alpha$.   \\

\begin{table}[h]
    \centering
    \begin{tabular}{cccccc}
        &$N(0,1)$ & $t_{8}$ & $t_{6}$ & $P(8,1)$ &  $P(6,1)$ \\ \hline
        $\mathfrak{p} = 2$ \\
         Interval Length& 0.0477 & 0.062 & 0.075 & 0.131 & 0.143 \\
         Coverage Rate  & 0.970 & 0.948 &  0.949 & 0.855 &0.828 \\ \hline
         $\mathfrak{p = 1.8}$  \\
          Interval Length & 0.110 &0.136 & 0.153 & 0.250 & 0.294  \\
          Coverage Rate & 0.991 &0.983& 0.985 & 0.943 & 0.924 \\ \hline
          $\mathfrak{p} = 1.65$ \\
          Interval Length & 0.151 & 0.184 & 0.192& 0.321 & 0.245\\
          Coverage Rate & 0.988 & 0.989 & 0.989 & 0.954 & 0.926 \\ \hline
          $\mathfrak{p} = 1.5$ \\
          Interval Length & 0.216  & 0.261 & 0.282 & 0.428 & 0.458 \\
          Coverage Rate &  0.999 & 0.990 & 0.994 & 0.971 & 968 \\ \hline
          $\mathfrak{p} = 1.35$ \\
          Interval Length & 0.353  & 0.400 & 0.415 & 0.603 & 0.630 \\
          Coverage Rate &  0.999 & 0.995 & 0.995 & 0.987 & 0.984 
          
    \end{tabular}
    \caption{Coverage rate and interval length for the proposed methods with distinct $\mathfrak{p}$. The simulations represent a 95$\%$ coverage rate where 1000 simulations were performed.}
    \label{tab:placeholder}
\end{table}

It appears that the standard known asymptotic limit, $\mathfrak{p} = 2$, leads to  a good enough  approximation for certain distributions, nonetheless, it seems to consistently belittle the confidence intervals whenever the innovation processes exhibit fat tails. The slower the tail decay behaviour, the worse the standard algorithm performs. \\\\
Via resampling this problem is solved, as it can be appreciated in the table. Decreasing the parameter $\mathfrak{p}$ leads to significantly better outcomes for fatter tailed distributions. Showing that, independently of the chosen innovations process there exist a $\mathfrak{p}$ which leads to the desired confidence intervals.

\vspace{1cm}

Marc Taberner Ortiz,  School of Computation, Information and Technology, Technische Universit\"at M\"unchen
\newline e-mail: marctaberner@gmail.com

Manfred Denker, Institut f\"ur Mathematische Stochastik, Universit\"at G\"ottingen
\newline e-mail: manfred.denker@mathematik.uni-goettingen.de

\end{document}